\documentclass[12pt,reqno]{amsart}
\usepackage{amsmath,amssymb,amsfonts,amsthm}
\usepackage[left=3cm, right=3cm, top=2.8cm, bottom=2.8cm]{geometry}

\usepackage{pstricks}                        
\usepackage{pstricks-add}               
\usepackage{pst-all,pst-math,pst-plot}       
\usepackage{textcomp}

\usepackage{tikz}
\usetikzlibrary{matrix}

\usepackage{graphicx,color}
\usepackage{graphicx}
\usepackage{amsmath}
\usepackage[all]{xy}
\usepackage{amsmath,amssymb}
\usepackage{amsthm}
\usepackage{amsfonts}
\usepackage{indentfirst}
\usepackage{geometry}
\usepackage[T1]{fontenc}
\usepackage{times}
\usepackage{color,graphicx}

\newtheorem{pr}{Proposition}

\newtheorem{lema}{Lemma}
\newtheorem{theorem}{Theorem}
\newtheorem{de}{Definition}

\newcommand{\N}{\mathbb{N}}	
\newcommand{\R}{\mathbb{R}}	
\newcommand{\C}{\mathbb{C}} 
\newcommand{\Le}{\mathcal{L}}

\makeindex

\begin{document}

	\title[ Fractional diffusion-wave equations]{Well-posedness and regularity theory for the fractional diffusion-wave equation in Lebesgue spaces}
	

	\author[B. de Andrade]{Bruno de Andrade}
	\address[B. de Andrade]{Departamento de Matem\'atica, Universidade Federal de Sergipe, S\~ao Crist\'ov\~ao-SE, CEP 49100-000, Brazil.}
	\email{bruno@mat.ufs.br}
	\author[N. dos Santos]{Naldisson dos Santos}
	\address[N. Santos]{Departamento de Matem\'atica, Universidade Federal de Sergipe, S\~ao Crist\'ov\~ao-SE, CEP 49100-000, Brazil.}
	\email{naldisson@gmail.com}
	\date{}

	\begin{abstract}
		This paper is dedicated to the study of the semilinear fractional diffusion-wave equation. We provide estimates on the families of linear operators related to the problem in the fractional power scale associated with the Laplace operator. Furthermore, we analyze the existence and uniqueness of local mild solutions and their possible continuation to a maximal interval of existence. We also consider the issue of spatial regularity and continuous dependence with respect to initial data. \\
		
		\
		
		
		
	\end{abstract}

	
	\maketitle

	\section{Introduction} \label{intro}

In this paper we consider the  nonlinear fractional diffusion-wave equation 
\begin{eqnarray}\label{wave}	 					
	\left\{\begin{array}{lrr}
		\partial_t^{\alpha}u  = \Delta u  + f(u), \ \mbox{in}\ (0,\infty)\times\Omega,\\	
		u  =  0, \ \mbox{in} \  [0,\infty)\times \partial\Omega,\\
		u(0,x)=u_0(x),\ u_t(0,x)=u_1(x),\ \mbox{in}\ \ \Omega,\\
	\end{array} \right.	
\end{eqnarray}
	where $\Omega$ is an open subset of  $\R^N$ with sufficiently smooth boundary $\partial\Omega$, $\alpha\in (1,2)$ and $\partial_t^{\alpha}$ is the Caputo fractional derivative. Furthermore,  $f:\mathbb{R}\to\mathbb{R}$ is a continuous function that verifies $f(0)=0$ and 
	\begin{eqnarray}\label{fc}
		|f(r)-f(s)|\leq C(|r|^{\rho-1}+|s|^{\rho-1})|r-s|,\quad \forall r,s\in\mathbb{R},
	\end{eqnarray}
	for some $C>0$ and $\rho>1$.
	We analyze the above problem in the $L^q(\Omega)$ setting, $1<q<\infty$.
This kind of evolution problem has been the subject of several research papers in the last years since this topic involves a large variety of natural sciences such as physics, chemistry, biology, engineering, and medicine, see Refs. \cite{Bouchaud,Fre, MK,MIMP, Uch} and the references therein.

An important example where these equations appear naturally is the theory of one-dimen- sion linear viscoelasticity. The basic equations of such a theory are given by			
\begin{equation}\label{basic1}
	\sigma_x(x,t)=\rho u_{tt}(x,t),
\end{equation}
\begin{equation}
	\epsilon(x,t)= u_{x}(x,t),
\end{equation}
\begin{equation}\label{basic3}
	\epsilon(x,t)= J_0\ \sigma(x,t) + \dot{J}(t)\ast \sigma(x,t),
\end{equation}
where $\rho$ is the density, $u$ denotes the displacement, $\sigma$ the stress and $\epsilon$ the strain. Furthermore, $J(t)$  represents the creep compliance, which is assumed to be a non-negative, non-decreasing function with initial value $J_0=J(0^+)\ge0$, called the instantaneous compliance. The evolution equation for the the displacement $u(x,t)$ can be derived from \eqref{basic1}-\eqref{basic3} since  that 
\begin{equation}\label{relation}
	u_{xx}=\epsilon_x=(J_0 + \dot{J}\ast)\sigma_x = (J_0 + \dot{J}\ast)\rho u_{tt}.
\end{equation}
If we consider the so-called power-type material, for which the creep compliance can be written as 
$$J(t)=\frac{1}{\rho a}\frac{t^\gamma}{\Gamma(\gamma+1)},\quad 0<\gamma\le 1,\quad t>0,$$
where $a$ is a positive constant and $\Gamma$ is the Gamma function, then from \eqref{relation} we obtain the evolution equation 
\begin{equation}\label{evo}
	\partial^{\alpha}_{t}u=a u_{xx},
\end{equation}
with $1\le \alpha=2-\gamma\le 2$, and $\partial^{\alpha}_{t}$ the time fractional derivative in Caputo's sense. Following Refs. \cite{CM,MT}, it follows that the above creep law is provided by viscoelastic models whose stress-strain relation verify
$$\sigma= \rho a \partial^{\gamma}_{t}\epsilon, \quad 0<\gamma\le 1.$$
For $\gamma=1$ we have the situation of a Newtonian fluid where $a$ represents the kinematic viscosity. In this case \eqref{evo} becomes the classical diffusion equation. In the limit case $\gamma=0$ we obtain from \eqref{evo} the classical D'Alembert wave equation with wave-front velocity $c=\sqrt{a}$. When $0<\gamma<1$, the evolution equation \eqref{evo} is called the fractional diffusion-wave equation and has been the subject of many research works, see e.g. Refs. \cite{ KKL, MK, MP} and references therein.  Particularly in Ref. \cite{MP}, the authors show that \eqref{evo} governs the propagation of stress waves in viscoelastic media which are of relevance in acoustics and seismology since their quality factor turns out to be independent of frequency. 			
			
From the mathematical point of view, the study of existence, uniqueness, continuous dependence of the solutions, the regularity theory (spatial or temporal), the analysis of asymptotic behavior, and the characterization of special solutions are quite prominent themes in the theory of fractional evolution equations, see e.g. Refs. \cite{deAnd, B1, B2, AF, AV, ACCM, F1, F2, H, KKL, MP}. The goal of this work is to study well-posedness results and spatial regularity theory to the semilinear fractional diffusion-wave equation \eqref{wave}.  As we know, an efficient approach to partial differential equations is closely connected to the concept of solutions to be adopted.  To set the sense of solution we are looking for, we formally apply the Laplace transform to \eqref{wave} and obtain 
\begin{eqnarray*}
\widehat{\partial_t^{\alpha}u}(\lambda,\cdot)=\Delta\widehat{u}(\lambda,\cdot)+ \widehat{f}(\lambda),
\end{eqnarray*}
where $\widehat{f}(\lambda)$ represent the Laplace transform of the nonlinear term $f(u)$.  The inverse Laplace transform yields
\begin{equation}\label{vp}
u(t,x)= E_{\alpha}(t)u_0(x)+S_{\alpha}(t )u_1(x)+\int_{0}^{t}R_{\alpha}(t-s)f(u(s,x))ds, 
\end{equation}
$t\ge0$ and $x\in\Omega$, where $E_{\alpha}(t)$, $S_{\alpha}(t)$ and $R_{\alpha}(t)$ are the inverse Laplace transforms of $$\lambda^{\alpha-1}(\lambda^{\alpha}-\Delta)^{-1},\quad \lambda^{\alpha-2}(\lambda^{\alpha}-\Delta)^{-1}\quad \mbox{and}\quad  (\lambda^{\alpha}-\Delta)^{-1},$$ respectively. Motivated by this argumentation and related literature, we adopt the following notion of solution.
			
\begin{de}\label{ms}
A continuous function $u: [0,\tau ] \rightarrow L^{q}(\Omega)$ that satisfies \eqref{vp} is called a mild solution to the problem \eqref{wave}.
\end{de}


Consider the Laplace operator $\Delta$ with Dirichlet boundary conditions in $Y_{q}^{0}=L^q(\Omega)$, $1<q<\infty$, and domain  $Y_{q}^{1}=W^{2,q}(\Omega)\cap W_{0}^{1,q}(\Omega)$. We know that the linear problem  
$$	\left\{\begin{array}{lrr}
	\partial_t^{\alpha}u  = \Delta u, \ \mbox{in}\ (0,\infty)\times\Omega,\\	
	u  =  0, \ \mbox{in} \  [0,\infty)\times \partial\Omega,\\
	u(0,x)=u_0(x)\ \mbox{in}\ \ \Omega,\\
\end{array} \right.	
$$
is well-posed for all $\alpha\in(0,1]$ and any $1<q<\infty$. This occurs because $-\Delta$  is a sectorial operator, and hence there exists $\phi_q\in (\frac{\pi}{2}, 
\pi)$  such that $$\mathbb{S}_{\phi_q}:=\{z\in\C: |\arg(z)|\leq\phi_q, z\neq 0\}\subset\rho(\Delta).$$
Therefore,  if $\lambda\in\mathbb{S}_{\phi_q}$ then for all $\alpha\in(0,1]$ we have 
$$\lambda^{\alpha}\in\mathbb{S}_{\phi_q}\subset\rho(\Delta)\quad \mbox{and}\quad \|(\lambda^\alpha-\Delta)^{-1}\|_{\mathcal{L}(Y_{q}^{0})}\leq C|\lambda|^{-\alpha}.$$
With this property we can use Laplace transform methods and  verify that $\Delta$ generates an analytical semigroup in $L^q(\Omega)$, for any $1<q<\infty$. However, the scenario changes in the case $\alpha\in(1,2)$ since the above property no longer holds. In this context, it is not true that  
$\lambda^{\alpha}\in\mathbb{S}_{\phi_q}\subset\rho(\Delta)$, when $\lambda\in\mathbb{S}_{\phi_q}$.
Thereby, the respective linear fractional diffusion-wave equation
$$	\left\{\begin{array}{lrr}
	\partial_t^{\alpha}u  = \Delta u, \ \mbox{in}\ (0,\infty)\times\Omega,\\	
	u  =  0, \ \mbox{in} \  [0,\infty)\times \partial\Omega,\\
		u(0,x)=u_0(x),\ u_t(0,x)=u_1(x),\ \mbox{in}\ \ \Omega,\\
\end{array} \right.	
$$
is not well-posed for all $\alpha\in(1,2)$ and any $1<q<\infty$. This is compatible with the limit case $\alpha=2$, where the above problem becomes the linear wave equation, which is, in a general way,  ill-posed in $L^{q}(\Omega)$, $1\le q<\infty$. Then, it is reasonable to think that superdiffusive problems\footnote{ Derivation order $\alpha\in (1,2)$.} are likely more mathematically challenging than subdiffusive ones\footnote{ Derivation order $\alpha\in(0,1)$.}.

In Section \ref{Lest} we investigate further the linear families associated with the Laplace operator in the superdifusive context. At the core of our analysis is the smoothing effect of the families $(E_{\alpha}(t))_{t\ge0}$, $(S_{\alpha}(t))_{t\ge0}$ and $(R_{\alpha}(t))_{t\ge0}$ in the scale $\{X_{q}^{\beta}\}_{\beta\in\R}$ of fractional powers spaces associated with $-\Delta$. We prove they are bounded continuous families of linear operators in $L^q(\Omega)$, and  if $0\le\theta<\beta\le 1$, then there exists $M>0$ such that
\begin{equation*}
	\|E_{\alpha}(t)x\|_{X^{1+\theta}_{q}}\leq Mt^{-\alpha(1+\theta-\beta)}\|x\|_{X^{\beta}_{q}},\quad \|S_{\alpha}(t)x\|_{X^{1+\theta}_{q}}\leq Mt^{1-\alpha(1+\theta-\beta)}\|x\|_{X^{\beta}_{q}}
\end{equation*}
and
\begin{equation*}
	\|R_{\alpha}(t)x\|_{X^{1+\theta}_{q}}\leq Mt^{-1-\alpha(\theta-\beta)}\|x\|_{X^{\beta}_{q}},
\end{equation*}
for all $x\in X^{\beta}_{q}$ and $t>0$. 	Particularly, it follows from the above estimate that for all $t>0$, the operators $t^{\alpha\theta}E_{\alpha}(t)$ from $X^{1}_{q}$ into $X^{1+\theta}_{q}$ are bounded linear operators satisfying
$$\|t^{\alpha\theta}E_{\alpha}(t)\|_{\mathcal{L}(X^{1}_{q},X^{1+\theta}_{q})}\le M,$$
with $M>0$ independent of $t>0$, for all $0\le\theta\le 1$. Moreover, given a compact subset $J$ of $X^{1}_{q}$, we have
$$\lim_{t\to0^{+}}\sup_{x\in J}\|t^{\alpha\theta}E_{\alpha}(t)x\|_{X^{1+\theta}_{q}}=0.$$
To prove the last assertion it is sufficient to observe that the operators $$t^{\alpha\theta}E_{\alpha}(t):X^{1}_{q}\to X^{1+\theta}_{q}$$ are uniformly bounded in $t>0$,  $\lim_{t\to0^{+}}\|t^{\alpha\theta}E_{\alpha}(t)x\|_{X^{1+\theta}_{q}}=0$, for  $x\in X^{1+\theta}_{q},$
and $X^{1+\theta}_{q}$ is a dense subset of $X^{1}_{q}$. Similar properties can be proved to $(S_{\alpha}(t))_{t\ge0}$ and $(R_{\alpha}(t))_{t\ge0}$.

With this smoothing effect, given $u_0,u_1\in L^{q}(\Omega)$, we provide sufficient conditions to the existence of a unique mild solution $u\in C([0,\tau ], L^{q}(\Omega))$ to \eqref{wave},	 			which can be continued to a maximal time of existence $\tau_{max}>0$ such that
$$\limsup_{t\rightarrow \tau_{max}^{-}}\int_{\Omega}|u(t,x)|^{q}dx=+\infty,$$
if $\tau_{max}<+\infty$. At this point, we want to highlight that the fractional powers scale has proven to be an effective tool for studying the issue of regularity. These spaces offer a precise method for measuring the spatial regularity of the solutions. Utilizing them, we demonstrate that the mild solution of \eqref{wave} exhibits an immediate regularization effect. Precisely, for all $0\le\theta<\beta$, we have that
$$u\in C((0,\tau ],X_{q}^{1+\theta}),$$
and, if $\theta>0$ then
$$\lim_{t\to0^{+}}t^{\alpha\theta}\|u(t)\|_{X_{q}^{1+\theta}}= 0.$$
Given that $X_{q}^{1+\theta}\hookrightarrow L^q(\Omega)$, it follows that $X_{q}^{1+\theta}$ is a more regular space than $L^q(\Omega)$. 
			
The structure of this work is as follows. In the next section, we rewrite the semilinear fractional diffusion-wave equation \eqref{wave} as an abstract evolution problem in a suitable scale of Banach spaces. In Section \ref{Lest}, we present a detailed analysis of the families $(E_{\alpha}(t))_{t\ge0}$, $(S_{\alpha}(t))_{t\ge0}$ and $(R_{\alpha}(t))_{t\ge0}$ in the scale of fractional powers spaces associated with the Laplace operator. Section \ref{Main} includes statements and proofs of our main results, along with a three-dimensional concrete example to illustrate their applicability.

\section{Abstract framework}			
The main goal of this section is to introduce and study spaces measuring smoothness which are also algebraically well-adapted to the problem \eqref{wave}. Remember that the Laplace operator with Dirichlet boundary conditions is a sectorial operator in $Y_{q}^{0}=L^q(\Omega)$ with domain  $Y_{q}^{1}=W^{2,q}(\Omega)\cap W_{0}^{1,q}(\Omega)$. Denoting by $\{Y_{q}^{\beta}\}_{\beta\in\R}$ the  fractional powers spaces associated with the Laplace operator, it follows that, see Ref. \cite{Amann2},
\begin{eqnarray*}
	\left\{\begin{array}{lcr}
		Y_{q}^{\beta}\hookrightarrow H_{q}^{2\beta}(\Omega), \ \ \beta\geq 0, \ \ 1<q<\infty,\\
		Y_{q}^{-\beta}=(Y_{q'}^{\beta})', \ \ \beta\geq 0, \ \ 1<q<\infty, \ \ q'=\frac{q}{q-1}.\\
	\end{array}\right.
\end{eqnarray*}
where  $H_{q}^{2\beta}(\Omega):=W^{2\beta,q}(\Omega)$ is the standard Sobolev-Slobodeckii space and $ (Y^{\beta}_{q})'$ represents the dual space of $Y^{\beta}_{q}$. Therefore, from standard duality arguments, we get
\begin{eqnarray}\label{embHeat}
	\left\{\begin{array}{lcr}
		Y_{q}^{\beta}\hookrightarrow L^r(\Omega), \ \ r\leq \frac{Nq}{N-2\beta q}, \ \ 0\leq\beta<\frac{N}{2q},\\
		Y_{q}^{0}=L^q(\Omega),\\
		Y_{q}^{\beta}\hookleftarrow L^s(\Omega), \ \ s\geq \frac{Nq}{N-2\beta q}, \ \ -\frac{N}{2q'}<\beta\leq 0.
	\end{array}\right.
\end{eqnarray}
Moreover, the realization of $L=\Delta$ in $Y_{q}^{\beta}$, denoted by $L_{\beta}$, is an isometry from $Y_{q}^{\beta+1}$ into $Y_{q}^{\beta}$ and 
$$L_{\beta}: D(L_{\beta})=E_{q}^{\beta+1}\subset E_{q}^{\beta}\rightarrow E_{q}^{\beta}$$ 
is a sector operator. In addition, $D(L_{\beta}^{\gamma})=Y_{q}^{\beta+\gamma}$.

We shall work the problem \eqref{wave} with initial data belonging to $L^q(\Omega)$. For this reason, we set $X_{q}^{\beta}:=Y_{q}^{\beta-1}$, $\beta\in\R$, and let $A_q: X_{q}^{1}\subset X_{q}^{0}\rightarrow X_{q}^{0}$ be  the operator $L_{-1}$. From \eqref{embHeat}, it follows that the fractional powers scale associated with $A_q$ satisfies 
\begin{eqnarray}\label{IS}
	\left\{\begin{array}{lrc}
		X_{q}^{\beta}\hookrightarrow L^r(\Omega), \ \ r\leq \frac{Nq}{N+2q-2\beta q}, \ \ 1\leq\beta<1+\frac{N}{2q},\\
		X_{q}^{1}=L^q(\Omega),\\
		X_{q}^{\beta}\hookleftarrow L^s(\Omega), \ \ s\geq \frac{Nq}{N+2q-2\beta q}, \ \ 1-\frac{N}{2q'}<\beta\leq 1.\\
	\end{array}\right.
\end{eqnarray}
Therefore, the abstract formulation of the problem $(\ref{wave})$ in the $L^q(\Omega)$ theory becomes  
\begin{eqnarray} \label{awave}
	\left\{\begin{array}{lrr}
		D^{\alpha}_{t}u=A_qu + f(u),\ \ t>0,\\
		u(0)=u_0, \ u'(0)=u_1,\\
	\end{array} \right.
\end{eqnarray}
for $\psi\in X_{q}^{1}=L^{q}(\Omega)$.
			
			\
			
We close this section proving some properties about function  $f$.

\begin{lema}\label{lemaonda}
	Consider $\beta \ge 0$ such that $1-\frac{N}{2q'}<\beta<1$ and $1<\rho\leq 1+\frac{2q}{N}(1-\beta )$, then $f:X_{q}^{1}\to X_{q}^{\beta}$ verifies 
	\begin{equation*}
		\|f(u)-f(v)\|_{X_{q}^{\beta}}\leq C (\|u\|_{X_{q}^{1}}^{\rho-1}+\|v\|_{X_{q}^{1}}^{\rho-1})\|u-v\|_{X_{q}^{1}}
	\end{equation*} 
	and
	\begin{equation*}
		\|f(u)\|_{X_{q}^{\beta}}\leq C \|u\|_{X_{q}^{1}}^{\rho},
	\end{equation*} 
	for some $C>0$ and every $u,v\in X_{q}^{1}$.
\end{lema}
\begin{proof}
	Using the assumption $1-\frac{N}{2q'}<\beta<1$ and $(\ref{IS})$, we obtain 
	\begin{equation}\label{ime1}
		L^{\frac{Nq}{N+2q-2\beta q}}(\Omega)\hookrightarrow X_{q}^{\beta}.
	\end{equation}
	Moreover, noting that $1<\rho\leq 1+\frac{2q}{N}(1-\beta)$ implies
	$\frac{\rho Nq}{N+2q-2\beta q}\leq q$, we conclude that
	\begin{equation}\label{ime2}
		L^{\frac{\rho Nq}{N+2q-2\beta q}}(\Omega)\hookleftarrow L^q(\Omega).
	\end{equation}
	Finally, taking $u,v\in X_q^1=L^q(\Omega)$ and  using $(\ref{ime1})$, $(\ref{ime2})$ and H\"older's inequality, we get
	\begin{eqnarray*}
		\|f(u)-f(v)\|_{X_{q}^{\beta}} &\leq& \|f(u)-f(v)\|_{L^{\frac{Nq}{N+2q-2\beta q}}(\Omega)}\\
		&\leq & C\|u-v\|_{L^{\frac{\rho Nq}{N+2q-2\beta q}}(\Omega)}(\|u\|_{L^{\frac{\rho Nq}{N+2q-2\beta q}}(\Omega)}^{\rho-1}+\|v\|_{L^{\frac{\rho Nq}{N+2q-2\beta q}}(\Omega)}^{\rho-1})\\
		&\leq& C\|u-v\|_{X_{q}^{1}}(\|u\|_{X_{q}^{1}}^{\rho-1}+\|v\|_{X_{q}^{1}}^{\rho-1})
	\end{eqnarray*}
Finally, taking $v=0$  we get
	\begin{eqnarray*}
		\|f(u)\|_{X_{q}^{\beta}} &\leq&  C\|u\|_{X_{q}^{1}}^{\rho}.
	\end{eqnarray*}
\end{proof}

\section{Linear estimates}			\label{Lest}
This section is dedicated to the study of  $(E_{\alpha}(t))_{t\ge0}$, $(S_{\alpha}(t))_{t\ge0}$ and $(R_{\alpha}(t))_{t\ge0}$ in the scale $X_{q}^{\beta}:=Y_{q}^{\beta-1}$, $\beta\in\R$, of fractional powers spaces associated with the sectorial operator $A_q: X_{q}^{1}\subset X_{q}^{0}\rightarrow X_{q}^{0}$ defined in the previous section.
Initially, consider $\phi_q\in (\frac{\pi}{2}, 
\pi)$  such that $\mathbb{S}_{\phi_q}:=\{z\in\C: |\arg(z)|\leq\phi_q, z\neq 0\}\subset\rho(A_q).$
From now on, we consider $\alpha\in (1,\frac{2\phi_q}{\pi})$. Note that if $\eta_0\in (\frac{\pi}{2}, \frac{\phi_q}{\alpha})$ and $\lambda\in\mathbb{S}_{\eta_0}$ then 
$$\lambda^{\alpha}\in\mathbb{S}_{\phi_q}\subset\rho(A_q)\quad \mbox{and}\quad \|(\lambda^\alpha-A_q)^{-1}\|_{\mathcal{L}(X_{q}^{1})}\leq C|\lambda|^{-\alpha}.$$

We start with the following result.

\begin{pr}\label{lf1} Consider $1<\alpha<\frac{2\phi_q}{\pi}$. For the function 
	\begin{equation}\label{Ealfa}
		E_{\alpha}(t):=\left\{\begin{array}{ll}\displaystyle \frac{1}{2\pi i}\int_{Ha}e^{\lambda t}\lambda^{\alpha-1}(\lambda^{\alpha}-A_{q})^{-1}d\lambda, &  \mbox{if} ~  t> 0,\\
			I, &  \mbox{if} ~ t=0,\end{array}\right.
	\end{equation}
	where $Ha$ is a suitable path and $I$ is the identity operator, the following properties are verified:
	\begin{itemize}
		\item[$(i)$] For all $t>0$, $E_{\alpha}(t):X^{0}_{q}\to X^{0}_{q}$ is well defined and there exists $M>0$ such that 
		$$\|E_{\alpha}(t)x\|_{X^{0}_{q}}\leq M\|x\|_{X^{0}_{q}}, \quad  \forall x\in X^{0}_{q}.$$ 
		\item[$(ii)$] The function $E_{\alpha}(\cdot):[0,\infty)\rightarrow \mathcal{L}(X^{0}_{q})$ is strongly continuous.
		\item[$(iii)$] If  $\beta\in[0, 1]$ then there exists a constant $M>0$ such that 
		$$\|E_{\alpha}(t)x\|_{X^{\beta}_{q}}\leq M t^{-\alpha\beta}\|x\|_{X^{0}_{q}},\quad \forall t>0, \ \forall x\in X^{0}_{q}.$$
	\end{itemize} 
\end{pr}

\begin{proof}
	\noindent $(i)$ For $1<\alpha<\frac{2\phi_q}{\pi}$ and $r>0$, set $\eta_0\in\left(\frac{\pi}{2} , \frac{\phi_q}{\alpha}\right)$ and define the path $Ha=Ha(r,\eta_0)$  by
	\begin{eqnarray}\label{Hankel}
		Ha=Ha(r,\eta_0)&:=& \{se^{i\eta_0}:r\leq s< \infty\}\cup\{re^{is}:|s|\leq \eta_0\}\cup \{se^{-i\eta_0}:r\leq s<\infty\}\nonumber\\
		& = & Ha_1+Ha_2-Ha_3.
	\end{eqnarray}

	For each $t>0$ fixed, if we assume $r=1/t$, then

		\noindent $\bullet$ Over $Ha_1$ we have
		\begin{eqnarray*}
			\left\|\frac{1}{2\pi i}\int_{Ha_1}e^{\lambda t}\lambda^{\alpha-1}(\lambda^{\alpha}-A_q)^{-1}xd\lambda\right\|_{X^{0}_{q}} 
			& \leq & \frac{C}{2\pi}\int_{1/t}^\infty e^{st\cos\eta_0}s^{-1}ds\|x\|_{X^{0}_{q}}\\
			& \leq & \frac{C}{2\pi}\int_{1/t}^\infty e^{st\cos\eta_0}tds\|x\|_{X^{0}_{q}}\\
			& = & \frac{C}{2\pi}\frac{e^{\cos\eta_0}}{|\cos\eta_0|}\|x\|_{X^{0}_{q}}.
		\end{eqnarray*}
		
		\noindent $\bullet$ Over $Ha_2$ we have
		\begin{eqnarray*}
			\left\|\frac{1}{2\pi i}\int_{Ha_2}e^{\lambda t}\lambda^{\alpha-1}(\lambda^{\alpha}-A_q)^{-1}xd\lambda\right\|_{X^{0}_{q}} 
			& \leq & \frac{C}{2\pi}\int_{-\eta_0}^{\eta_0} e^{\cos s}r^{\alpha-1} r^{-\alpha}rds\|x\|_{X^{0}_{q}}\\
			& \leq & \frac{C}{\pi}e\eta_0 \|x\|_{X^{0}_{q}}.
		\end{eqnarray*}
		
		\noindent $\bullet$ Over $Ha_3$ we proceed in the same way as in $Ha_1.$

	Taking $M>0$ as the sum of all bounds obtained above, we deduce that $E_\alpha(t)$ is well defined for each $t>0$ and
	$$\|E_{\alpha}(t)x\|_{X^{0}_{q}}\leq M\|x\|_{X^{0}_{q}}, \ \ \forall t>0, \ \ \forall x\in X^{0}_{q}.$$ 
	
	A classical argumentation using the Cauchy integral theorem ensures that the representation \eqref{Ealfa} is independent of $r>0$ and $\eta_0\in\left(\frac{\pi}{2},\frac{\phi_q}{\alpha}\right)$. Hence, this completes the proof of $(i)$.


		\noindent $(ii)$ Consider $\tau>0$ and let $\{t_n\}_{n\in\mathbb{N}}\subset (0,\tau)$ be such that $t_n\rightarrow \tau$ as $n\rightarrow\infty.$ Note that for $x\in X^{0}_{q}$
		\[\|E_\alpha(\tau)x-E_\alpha(t_n)x\|_{X^{0}_{q}} = \frac{1}{2\pi}\left\|\int_{Ha}(e^{\lambda\tau}-e^{\lambda t_n})\lambda^{\alpha-1}(\lambda^{\alpha}-A_q)^{-1}xd\lambda \right\|_{X^{0}_{q}}.\]
		
		We will analyze the above integral over the paths $Ha_1,$ $Ha_2$ and $Ha_3$ separately.

		\noindent $\bullet$ Over $Ha_1$ we have that
		\[\left\|\int_{Ha_1}(e^{\lambda\tau}-e^{\lambda t_n})\lambda^{\alpha-1}(\lambda^{\alpha}-A_q)^{-1}xd\lambda \right\|_{X^{0}_{q}} \leq  C \int_r^\infty \left|e^{\tau se^{i\eta}}-e^{t_n se^{i\eta}}\right|s^{-1}ds\|x\|_{X^{0}_{q}}.\]
		From the dominated convergence theorem, this last term goes to $0$ as $n$ goes to $\infty$.
		
		\noindent $\bullet$ Over $Ha_2$ we have that
		\[\left\|\int_{Ha_2}(e^{\lambda\tau}-e^{\lambda t_n})\lambda^{\alpha-1}(\lambda^{\alpha}-A_q)^{-1}xd\lambda \right\|_{X^{0}_{q}} \leq  C \int_{-\eta}^\eta \left|e^{\tau re^{is}}-e^{t_n re^{is}}\right|ds\|x\|_{X^{0}_{q}}.\]
		Again, from the dominated convergence theorem, this last term goes to $0$ as $n$ goes to $\infty$.
		
		\noindent $\bullet$ Over $Ha_3$ we proceed in the same way as in $Ha_1.$

		Therefore 
		$$\|E_\alpha(\tau )x-E_\alpha(t_n )x\|_{X^{0}_{q}}\rightarrow 0$$ 
		as $n\rightarrow\infty$, for each $x\in X^{0}_{q}$. Analogously, let $\{t_n\}_{n\in\mathbb{N}}\subset (\tau,\infty)$ such that $t_n\rightarrow \tau$ as $n\rightarrow\infty,$ then 
		$$\|E_\alpha(t_n )x-E_\alpha(\tau )x\|_{X^{0}_{q}}\rightarrow 0$$ 
		as $n\rightarrow\infty$, for each $x\in X^{0}_{q}$.
		
		For the proof of the continuity in $t=0$, note that by the Cauchy integral theorem
		\[I=\frac{1}{2\pi i}\left(\int_{Ha}e^\lambda\lambda^{-1}d\lambda\right) I,\]
		where $I$ denotes the identity operator. Then, for each $t>0$ and $x\in D(A_q)=X^{1}_{q}$ 
		\begin{equation*}
			\|E_\alpha(t)x-x\|_{X^{1}_{q}}  = \frac{1}{2\pi}\left\|\int_{Ha}e^{\lambda t}\lambda^{-1}(\lambda^{\alpha}-A_q)^{-1}A_q xd\lambda\right\|_{X^{0}_{q}}.
		\end{equation*}
		If we assume that $r=1/t$ then

		\noindent $\bullet$  Over $Ha_1$ we have 
		\begin{equation*}
			\frac{1}{2\pi}\left\|\int_{Ha_1}e^{\lambda t}\lambda^{-1}(\lambda^{\alpha}-A_q)^{-1}A_q xd\lambda\right\|_{X^{0}_{q}} \leq  \frac{Ce^{\cos\eta}t^{\alpha+1}}{2\pi|\cos\eta|}\|A_q x\|_{X^{0}_{q}}\to 0,\ \mbox{as}\  t\rightarrow 0^+.
		\end{equation*}
		
		\noindent $\bullet$  Over $Ha_2,$ we have 
		\begin{equation*}
			\frac{1}{2\pi}\left\|\int_{Ha_2}e^{\lambda t}\lambda^{-1}(\lambda^{\alpha}-A_q)^{-1}A_q xd\lambda\right\|_{X^{0}_{q}}  \leq  \frac{Ct^{\alpha}e\eta}{\pi}\|A_q x\|_{X^{0}_{q}}\to 0,\ \mbox{as}\  t\rightarrow 0^+.
		\end{equation*}

		\noindent $\bullet$  Over $Ha_3$ we proceed in the same way as in $Ha_1.$
		
		Therefore  $\|E_\alpha(t)x-x\|_{X^{1}_{q}}\rightarrow 0$,
		as $t\rightarrow 0 ^+,$ for any $x\in D(A_{q}).$ Since $X^{0}_{q}=\overline{D(A_{q})},$ it follows that $$\|E_\alpha(t)x-x\|_{X^{0}_{q}}\rightarrow 0\ \mbox{as}\ t\rightarrow 0 ^+,$$ for any $x\in X^{0}_{q}$. This ends the proof of $(ii)$.
		

		\noindent $(iii)$ 
		We consider initially the case $\beta =1 $ and we make use of similar argument already used above. Let $Ha$ be the path defined by (\ref{Hankel}), then
		
		\[\|E_{\alpha}(t)x\|_{X^{1}_{q}}=\frac{1}{2\pi}\left\|\int_{Ha}e^{\lambda t}\lambda^{\alpha-1} A(\lambda^{\alpha}-A)^{-1}xd\lambda\right\|_{X^{0}_{q}}.\]
		For each fixed $t>0$, if we assume that $r=1/t$ then

		\noindent $\bullet$  Over $Ha_1$ there exists a constant $k>0$ independent of $\alpha$ such that
		\begin{equation*}
			\frac{1}{2\pi}\left\|\int_{Ha_1}e^{\lambda t}\lambda^{\alpha-1} A(\lambda^{\alpha}-A)^{-1}xd\lambda\right\|_{X^{0}_{q}} \leq \frac{k}{2\pi}\int_{1/t}^\infty e^{st\cos\eta} s^{\alpha-1}ds\|x\|_{X^{0}_{q}}.
		\end{equation*}
		Since $\alpha-1 > 0,$ it follows that
		\begin{eqnarray*}
			\frac{k}{2\pi}\int_{1/t}^\infty e^{st\cos\eta} s^{\alpha-1}ds & = & \frac{k}{2\pi}\frac{e^{\cos\eta}}{|\cos\eta|}t^{-\alpha} -\frac{k}{2\pi}\frac{(\alpha-1)}{t\cos\eta}\int_{1/t}^\infty e^{st\cos\eta}s^{\alpha-2}ds\\
			&\leq & \frac{k}{2\pi}\frac{e^{\cos\eta}}{|\cos\eta|}t^{-\alpha} +\frac{k(\alpha-1)e^{\cos\eta}}{2\pi|\cos\eta|^2}t^{-\alpha}\\
			& < & \frac{k}{2\pi|\cos\eta|}\left(1+\frac{1}{|\cos\eta |}\right)t^{-\alpha}.
		\end{eqnarray*}
		
		\noindent $\bullet$  Over $Ha_2$ we have that
		\begin{equation*}
			\frac{1}{2\pi}\left\|\int_{Ha_2}e^{\lambda t}\lambda^{\alpha-1} A(\lambda^{\alpha}-A)^{-1}xd\lambda\right\|_{X^{0}_{q}} \leq  \frac{k}{2\pi}t^{-\alpha}\int_{-\eta}^\eta e^{\cos s}ds \|x\|_{X^{0}_{q}}
			\leq  ket^{-\alpha}\|x\|_{X^{0}_{q}}.
		\end{equation*}
		
		\noindent $\bullet$  Over $Ha_3$ we proceed in the same way as in $Ha_1.$
		
Taking $M>0$ as the sum of the quotas obtained above, we conclude that
		$$\|E_{\alpha}(t)x\|_{X^{1}_{q}}\leq M t^{-\alpha}\|x\|_{X^{0}_{q}},\quad \forall t>0, \ \forall x\in X^{0}_{q}.$$
		From this fact and item $(i)$, we conclude the proof  using \cite[Theorem V.1.5.2]{Amann1}.
	\end{proof}
	
	\begin{pr}\label{analitic}
		Consider $1<\alpha<\frac{2\phi_q}{\pi}$. The family $(E_\alpha(t))_{t\ge0}$ defined by \eqref{Ealfa} admits an analytic extension to the sector $\mathbb{S}_{\pi-\phi_q}.$
	\end{pr}
	\begin{proof}
		For $t\in\mathbb{S}_{\pi-\phi_q}$ define $E_\alpha(t)$ by \eqref{Ealfa} and let $r=1/|t|.$ Then for each $x\in X^{0}_{q}$ we have
		\[\|E_\alpha(t)x\|_{X^{0}_{q}}\leq \left(\frac{C}{2\pi}\int_{Ha}e^{\mathrm{Re}(\lambda t)}|\lambda|^{-1}d|\lambda|\right)\|x\|_{X^{0}_{q}}.\]
		It is not hard to check that the integral of the right side in the above inequality is convergent. This estimate ensures the absolutely convergence of \eqref{Ealfa} in the sector $\mathbb{S}_{\pi-\phi_q}$, hence $E_\alpha(t)$ is analytic in this region.
	\end{proof}

	\begin{pr}\label{lf2} Consider $1<\alpha<\frac{2\phi_q}{\pi}$. For the function 
		\begin{equation*}
			S_{\alpha}(t):=\left\{\begin{array}{ll}\displaystyle \frac{1}{2\pi i}\int_{Ha}e^{\lambda t}\lambda^{\alpha-2}(\lambda^{\alpha}-A_{q})^{-1}d\lambda, &  \mbox{if} ~  t> 0,\\
				0, &  \mbox{if} ~ t=0,\end{array}\right.
		\end{equation*}
		where $Ha$ is a suitable path and $0$ is the null operator, the following properties are verified:
		\begin{itemize}
			\item[$(i)$] For all $t>0$, $S_{\alpha}(t):X^{0}_{q}\to X^{0}_{q}$ is well defined and there exists $M>0$ such that 
			$$\|S_{\alpha}(t)x\|_{X^{0}_{q}}\leq Mt\|x\|_{X^{0}_{q}}, \quad  \forall x\in X^{0}_{q}.$$ 
			\item[$(ii)$] The function $S_{\alpha}(\cdot ):[0,\infty)\rightarrow \Le(X^{0}_{q})$ is strongly continuous.
			\item[$(iii)$] If  $\beta\in[0, 1]$ then there exists a constant $M>0$ such that 
			$$\|S_{\alpha}(t)x\|_{X^{\beta}_{q}}\leq M t^{1-\alpha\beta}\|x\|_{X^{0}_{q}},\quad \forall t>0, \ \forall x\in X^{0}_{q}.$$
			\item[$(iv)$] The family $(S_\alpha(t))_{t\ge0}$ admits an analytic extension to the sector $\mathbb{S}_{\pi-\phi_q}.$
		\end{itemize} 
	\end{pr}
	
	\begin{proof}
		In fact, using the uniqueness of the Laplace transform we have that 
		$$S_{\alpha}(t)x=\int_{0}^{t}E_{\alpha}(s)xds,$$
		for all $t\ge0$ and $x\in X^{0}_{q}$. Hence, the proof follows from Propositions \ref{lf1} and \ref{analitic}.
	\end{proof}

	\begin{pr}\label{lf3} Consider $1<\alpha<\frac{2\phi_q}{\pi}$. For the function 
		\begin{equation*}
			R_{\alpha}(t):=\left\{\begin{array}{ll}\displaystyle \frac{1}{2\pi i}\int_{Ha}e^{\lambda t}(\lambda^{\alpha}-A_q)^{-1}d\lambda, &  \mbox{if} ~  t> 0,\\
				0, &  \mbox{if} ~ t=0,\end{array}\right.
		\end{equation*}
		where $Ha$ is a suitable path and $0$ is the null operator, the following properties are verified:
		\begin{itemize}
			\item[$(i)$] For all $t>0$, $R_{\alpha}(t):X^{0}_{q}\to X^{0}_{q}$ is well defined and there exists $M>0$ such that 
			$$\|R_{\alpha}(t)x\|_{X^{0}_{q}}\leq Mt^{\alpha-1}\|x\|_{X^{0}_{q}}, \quad  \forall x\in X^{0}_{q}.$$ 
			\item[$(ii)$] The function $R_{\alpha}(\cdot ):[0,\infty)\rightarrow \Le(X^{0}_{q})$ is strongly continuous.
			\item[$(iii)$] If  $\beta\in[0, 1]$ then there exists a constant $M>0$ such that 
			$$\|R_{\alpha}(t)x\|_{X^{\beta}_{q}}\leq M t^{\alpha(1-\beta)-1}\|x\|_{X^{0}_{q}},\quad \forall t>0, \ \forall x\in X^{0}_{q}.$$
			\item[$(iv)$] The family $(R_\alpha(t))_{t\ge0}$ admits an analytic extension to the sector $\mathbb{S}_{\pi-\phi_q}.$
		\end{itemize} 
	\end{pr}
	
	\begin{proof}
		As in the proof of the last proposition,  the uniqueness of the Laplace transform ensures that
		$$R_{\alpha}(t)x=\int_{0}^{t}g_{\alpha-1}(t-s)E_{\alpha}(s)xds,$$
		for all $t\ge0$ and $x\in X^{0}_{q}$. Hence, the result follows from Propositions \ref{lf1} and \ref{analitic}.
	\end{proof}

	\begin{pr}\label{compinter}
		Consider $1<\alpha<\frac{2\phi_q}{\pi}$ and $0\le\theta<\beta\le 1$. There exists $M>0$ such that
		\begin{equation*}
			\|E_{\alpha}(t)x\|_{X^{1+\theta}_{q}}\leq Mt^{-\alpha(1+\theta-\beta)}\|x\|_{X^{\beta}_{q}},\quad \|S_{\alpha}(t)x\|_{X^{1+\theta}_{q}}\leq Mt^{1-\alpha(1+\theta-\beta)}\|x\|_{X^{\beta}_{q}}
		\end{equation*}
		and
		\begin{equation*}
			\|R_{\alpha}(t)x\|_{X^{1+\theta}_{q}}\leq Mt^{-1-\alpha(\theta-\beta)}\|x\|_{X^{\beta}_{q}},
		\end{equation*}
		for all $x\in X^{\beta}_{q}$ and $t>0$. 
	\end{pr}
	
	\begin{proof} 
		Initially, note that
		$$\|E_{\alpha}(t)\|_{\mathcal{L}(X^{0}_{q},X^{1}_{q})}\leq Mt^{-\alpha}\quad\mbox{and}\quad \|E_{\alpha}(t)\|_{\mathcal{L}(X^{1}_{q},X^{1}_{q})}\leq M,$$
		for some $M>0$ and all $t>0$. Consequently, using \cite[Inequality I.2.1.1]{Amann1}, there exists $c>0$ such that
		$$\|E_{\alpha}(t)\|_{\mathcal{L}(X^{\beta}_{q},X^{1}_{q})}\leq c \|E_{\alpha}(t)\|_{\mathcal{L}(X^{0}_{q},X^{1}_{q})}^{1-\beta}\|E_{\alpha}(t)\|_{\mathcal{L}(X^{1}_{q},X^{1}_{q})}^\beta\le M t^{-\alpha(1-\beta)},$$
		for some $M>0$ and all $t>0$, which proves the case $\theta=0$.
		
		$$	\begin{tikzpicture}
			\matrix (m) [matrix of math nodes,row sep=5em,column sep=6em,minimum width=4em]
			{
				X^{1}_{q} & X^{\beta}_{q} & X^{0}_{q} \\
				X^{1}_{q} & X^{1}_{q} & X^{1}_{q} \\};
			\path[-stealth]
			(m-1-1) edge node [left] {$E_{\alpha}(t)$} (m-2-1) 
			edge  (m-1-2)
			(m-1-2)	edge  (m-1-3)
			(m-2-1) edge (m-2-2)
			(m-1-2) edge node [right] {$E_{\alpha}(t)$} (m-2-2)
			(m-2-2)	edge  (m-2-3)
			(m-1-3) edge node [right] {$E_{\alpha}(t)$} (m-2-3);
		\end{tikzpicture}$$
		
		On the other hand, the same steps used in the proof of item $(iii)$ of Proposition \ref{lf1} ensure that
		$$\|E_{\alpha}(t)\|_{\mathcal{L}(X^{1}_{q},X^{2}_{q})}\leq Mt^{-\alpha},$$
		for some $M>0$ and all $t>0$.  So, as before,
		there exists $c>0$ such that
		$$\|E_{\alpha}(t)\|_{\mathcal{L}(X^{\beta}_{q},X^{1+\beta}_{q})}\leq c \|E_{\alpha}(t)\|_{\mathcal{L}(X^{0}_{q},X^{1}_{q})}^{1-\beta}\|E_{\alpha}(t)\|_{\mathcal{L}(X^{1}_{q},X^{2}_{q})}^\beta\le M t^{-\alpha},$$
		for some $M>0$ and all $t>0$. 
		
		$$	\begin{tikzpicture}
			\matrix (m) [matrix of math nodes,row sep=5em,column sep=6em,minimum width=4em]
			{
				X^{1}_{q} & X^{1}_{q} & X^{0}_{q} \\
				X^{2}_{q} & X^{1+\beta}_{q} & X^{1}_{q} \\};
			\path[-stealth]
			(m-1-1) edge node [left] {$E_{\alpha}(t)$} (m-2-1) 
			edge  (m-1-2)
			(m-1-2)	edge  (m-1-3)
			(m-2-1) edge (m-2-2)
			(m-1-2) edge node [right] {$E_{\alpha}(t)$} (m-2-2)
			(m-2-2)	edge  (m-2-3)
			(m-1-3) edge node [right] {$E_{\alpha}(t)$} (m-2-3);
		\end{tikzpicture}$$
		Furthermore, since $\theta<\beta$ we have $1< 1+\theta<1+\beta$, for $\theta>0$. Thus, using \cite[Theorem V.1.5.2]{Amann1}, there exist a constant $c>0$ such that for all $x\in X^{\beta}_{q}$
		$$
		\|E_{\alpha}(t)x\|_{X^{1+\theta}_{q}}  \leq  c\|E_{\alpha}(t)x\|_{X^{1}_{q}}^{\frac{\beta-\theta}{\beta}}\|E_{\alpha}(t)x\|_{X^{1+\beta}_{q}}^{\frac{\theta}{\beta}} \le Mt^{-\alpha(1+\theta-\beta)}\|x\|_{X^{\beta}_{q}},
		$$
		for some $M>0$ and  all $t>0$. We conclude the proof using  the subordinations given in the proofs of Propositions \ref{lf2} and \ref{lf3} to obtain the results about $(S_{\alpha}(t))_{t\ge0}$ and $(R_{\alpha}(t))_{t\ge0}$, respectively.
	\end{proof}		
	

\section{Main results} \label{Main}

This section contains statements and proofs of our theorems. We start with the following result that ensures sufficient conditions for local existence, uniqueness, and spatial regularity to the problem.

\begin{theorem}\label{existence1}
	Let $1<\alpha<\frac{2\phi_q}{\pi}$. Consider $\beta \ge 0$ such that $1-\frac{N}{2q'}<\beta<1$ and $1<\rho\leq 1+\frac{2q}{N}(1-\beta )$. Given $v_0\in L^{q}(\Omega)$, we can consider $r>0$ and $\tau >0$ such that for any $u_0, u_1\in B_{r}(v_0)\subset L^{q}(\Omega)$ there exists a unique mild solution $u\in C([0,\tau ], L^{q}(\Omega))$ of the problem $(\ref{awave})$. Furthermore, for all $0\le\theta<\beta$, it follows that
	$$u\in C((0,\tau ],X_{q}^{1+\theta})$$
	and if $\theta>0$ then
	$$\lim_{t\to0}t^{\alpha\theta}\|u(t,u_0,u_1)\|_{X_{q}^{1+\theta}}= 0.$$
	Moreover, if $u_0, w_0, u_1, w_1\in B_r(v_0)\subset L^{q}(\Omega)$, then there exists a constant $c>0$ such that
	$$t^{\alpha\theta}\|u(t,u_0,u_1)-u(t,w_0,w_1)\|_{X_{q}^{1+\theta}}\le c\left(\|u_0 - w_0\|_{L^{q}(\Omega)}+\|u_1 - w_1\|_{L^{q}(\Omega)}\right),$$
	for all $t\in[0, \tau]$.
\end{theorem}
\begin{proof}
	Consider $0<\nu<1$ and $0<\tau<1$ such that 
	$$\|E_{\alpha}(t)v_0-v_0\|_{X^{1}_{q}}<\dfrac{\nu}{4},\quad \|S_{\alpha}(t)u_1\|_{X^{1}_{q}}<\dfrac{\nu}{4}$$
	and
	$$CMR\tau^{\alpha\beta}{\bf B}(\alpha\beta,1)<\dfrac{\nu}{4},$$
	for all $t\in[0,\tau]$,	where 
	$$R=\max\{(\|v_0\|_{X^{1}_{q}}+ \nu)^{\rho}, 2(\|v_0\|_{X^{1}_{q}}+ \nu)^{\rho-1}\}.$$
	For $r=\dfrac{\nu}{4M}$, set
	$$K=\left\{u\in C([0,\tau], X^{1}_{q}): \sup_{t\in [0,\tau]}\|u(t)-v_0\|_{X^{1}_{q}}\leq\nu\right\}$$
	and define the operator
	$$(Tu)(t)=E_{\alpha}(t)u_0+S_{\alpha}(t)u_1+\int_{0}^{t}R_{\alpha}(t-s)f(u(s))ds.$$
	Our purpose is to show that $K$ is a $T$-invariant set and $T$ is a contraction. Initially, let us prove that $T:K\to K$ is well defined. Consider $0\le t_1<t< \tau$ and $u\in K$. Then
	\begin{eqnarray*}
		\|(Tu)(t)-(Tu)(t_1)\|_{X^{1}_{q}}&\leq& \|E_{\alpha}(t)u_0-E_{\alpha}(t_1)u_0\|_{X^{1}_{q}}+ \|S_{\alpha}(t)u_1-S_{\alpha}(t_1)u_1\|_{X^{1}_{q}}\\
		&+& I,
	\end{eqnarray*}
	where $$I:=\left\|\int_{0}^{t}R_{\alpha}(t-s)f(u(s))ds-\int_{0}^{t_1}R_{\alpha}(t_1-s)f(u(s))ds\right\|_{X^{1}_{q}}.$$
Using the strong continuity of the families $(E_{\alpha}(t))_{t\ge0}$ and $(S_{\alpha}(t))_{t\ge0}$, we have that
	$$\|E_{\alpha}(t)u_0-E_{\alpha}(t_1)u_0\|_{X^{1}_{q}}\rightarrow 0 \quad \mbox{and} \quad \|S_{\alpha}(t)u_1-S_{\alpha}(t_1)u_1\|_{X^{1}_{q}}\rightarrow 0$$
	when $t\rightarrow t_1^+$. On the other hand
	\begin{eqnarray*}
		I&\leq&\int_{0}^{t_1}\|R_{\alpha}(t-s)f(u(s))-R_{\alpha}(t_1-s)f(u(s))\|_{X^{1}_{q}}ds\\
		&+& \int_{t_1}^{t}\|R_{\alpha}(t-s)f(u(s))\|_{X^{1}_{q}}ds.
	\end{eqnarray*} 
	Note that
	\begin{eqnarray*}
		\int_{t_1}^{t}\|R_{\alpha}(t-s)f(u(s))\|_{X^{1}_{q}}ds&\leq& M\int_{t_1}^{t}(t-s)^{\alpha\beta-1}\|f(u(s))\|_{X^{\beta}_{q}}ds\\
		&\leq& \dfrac{MC}{\alpha\beta}(\|v_0\|_{X^{1}_{q}}+\nu)^\rho(t-t_1)^{\alpha\beta}\rightarrow 0,
	\end{eqnarray*}
	when $t\rightarrow t_1^+.$  By setting
	$$I_1:=\int_{0}^{t_1}\|R_{\alpha}(t-s)f(u(s))-R_{\alpha}(t_1-s)f(u(s))\|_{X^{1}_{q}}ds,$$
	we obtain
	\begin{eqnarray*}
		I_1&\leq&\int_{0}^{t_1}[(t-s)^{\alpha\beta-1}-(t_1-s)^{\alpha\beta-1}]\|\tilde{R}_{\alpha}(t-s)f(u(s))\|_{X^{1}_{q}}ds\\
		&+& \int_{0}^{t_1}(t_1-s)^{\alpha\beta-1}\|[\tilde{R}_{\alpha}(t-s)f(u(s))-\tilde{R}_{\alpha}(t_1-s)f(u(s))\|_{X^{1}_{q}}ds,
	\end{eqnarray*}
	where $\tilde{R}_{\alpha}(t):= t^{1-\alpha\beta}R_{\alpha}(t)$, $t>0$. The family $(\tilde{R}_{\alpha}(t))_{t>0}$ is strongly continuous in $X^{0}_{q}$, and from Proposition \ref{compinter} we have
	\begin{equation*}
		\|\tilde{R}_{\alpha}(t)x\|_{X^{1}_{q}}\leq M \|x\|_{X^{\beta}_{q}}, \ \ \forall t>0,\ \ \forall x\in X^{\beta}_{q}.
	\end{equation*}
	Set
	$$I_2=\int_{0}^{t_1}[(t-s)^{\alpha\beta-1}-(t_1-s)^{\alpha\beta-1}]\|\tilde{R}_{\alpha}(t-s)f(u(s))\|_{X^{1}_{q}}ds$$
	and
	$$I_3=\int_{0}^{t_1}(t_1-s)^{\alpha\beta-1}\|[\tilde{R}_{\alpha}(t-s)f(u(s))-\tilde{R}_{\alpha}(t_1-s)f(u(s))\|_{X^{1}_{q}}ds.$$
	Then, 
	\begin{eqnarray*}
		I_2	&\leq& \dfrac{MC}{\alpha\beta}(\|v_0\|_{X^{1}_{q}}+\nu)^{\rho}[t^{\alpha\beta}-(t-t_1)^{\alpha\beta}-t_1^{\alpha\beta}]
		\rightarrow 0,
	\end{eqnarray*}
	when $t\rightarrow t_1^+$. Finally, using the strong continuity of $(\tilde{R}_{\alpha}(t))_{t>0}$ and the Lebesgue's dominated
	convergence theorem we prove that $I_3$ goes to $0$, as $t\rightarrow t_1^+$.
	This fact ensures that
	$$\|(Tu)(t)-(Tu)(t_1)\|_{X^{1}_{q}}\rightarrow 0,$$
	when $t\rightarrow t_1^+$. A similar argument proves the case $t_1 \in(0,\tau]$ and  $t\rightarrow t_1^-$.
	
	To conclude that $K$ is a $T$-invariant set, consider $t\in [0,\tau]$ and $u\in K$. Then we have
	\begin{eqnarray*}
		\|(Tu)(t)-v_0\|_{X^{1}_{q}}&\leq& \|E_{\alpha}(t)u_0-v_0\|_{X^{1}_{q}} + \|S_{\alpha}(t)u_1\|_{X^{1}_{q}}\\
		&+& \int_{0}^{t}\|R_{\alpha}(t-s)f(u(s))\|_{X^{1}_{q}} ds\\
		&\leq& \|E_{\alpha}(t)u_0-E_{\alpha}(t)v_0\|_{X^{1}_{q}}+\|E_{\alpha}(t)v_0-v_0\|_{X^{1}_{q}}\\
		&+& \|S_{\alpha}(t)u_1\|_{X^{1}_{q}} + CM\int_{0}^{t}(t-s)^{-1+\alpha\beta}\|u(s)\|_{X^{1}_{q}}^{\rho}ds\\
		&\leq&  \dfrac{3\nu}{4}+ CMR\tau^{\alpha\beta}B(1, \alpha\beta)\le \nu.
	\end{eqnarray*}
	That is, $T: K\to K$ is well defined.
	
	Now, let $t\in [0,\tau]$ and $u, v \in K$. Then
	$$\|(Tu)(t)-(Tv)(t)\|_{X^{1}_{q}}\leq \int_{0}^{t}\|R_{\alpha}(t-s)[f(u(s))-f(v(s))]\|_{X^{1}_{q}}ds$$
	$$\leq\ M\int_{0}^{t}(t-s)^{-1+\alpha\beta}\|f(u(s))-f(v(s))\|_{X^{\beta}}ds$$
	$$\leq\ MC\int_{0}^{t}(t-s)^{-1+\alpha\beta}[\|u(s)\|_{X^{1}_{q}}^{\rho-1}+ \|v(s)\|_{X^{1}_{q}}^{\rho-1}]\|u(s)-v(s)\|_{X^{1}_{q}}ds$$
	$$\leq\ MC[2(\|v_0\|_{X^{1}_{q}} +\nu)^{\rho-1}]\tau^{\alpha\beta}B(1,\alpha\beta)\sup_{s\in [0,\tau]}\|u(s)-v(s)\|_{X^{1}_{q}}$$
	$$\leq\ MCR\tau^{\alpha\beta}B(1,\alpha\beta)\sup_{s\in [0,\tau]}\|u(s)-v(s)\|_{X^{1}_{q}}< \frac{1}{4}\sup_{s\in [0,\tau]}\|u(s)-v(s)\|_{X^{1}_{q}}.$$
	This shows that $T:K\rightarrow K$ is a $\frac{1}{4}$-contraction and by the Banach fixed point theorem, $T$ has a unique fixed point in $K$ which is a mild solution to the problem $\eqref{awave}$. 
	
	To prove the uniqueness of mild solutions, suppose that $ u, v $ are two mild solutions of $(\ref{awave})$. Then for all $t\in [0,\tau]$ we have that
	$$\|u(t)-v(t)\|_{X^{1}_{q}}
	\leq\int_{0}^{t}\|R_{\alpha}(t-s)[f(u(s))-f(v(s))]\|_{X^{1}_{q}}ds$$
	$$\leq M \int_{0}^{t}(t-s)^{-1+\alpha\beta}\|f(u(s))-f(v(s))]\|_{X^{\beta}_{q}}ds$$
	$$\leq CM\int_{0}^{t}(t-s)^{-1+\alpha\beta}(\|u(s)\|_{X^{1}_{q}}^{\rho-1}+ \|v(s)\|_{X^{1}_{q}}^{\rho-1})\|u(s)-v(s)\|_{X^{1}_{q}}ds$$
	$$\leq b\int_{0}^{t}(t-s)^{-1+\alpha\beta}\|u(s)-v(s)\|_{X^{1}_{q}}ds,$$
	where $b:= CM(\sup_{s\in[0,\tau]}\|u(s)\|_{X^{1}_{q}}^{\rho-1}+ \sup_{s\in[0,\tau]}\|v(s)\|_{X^{1}_{q}}^{\rho-1}+1)$.
	We can now apply the singular Gronwall inequality, see \cite[Lemma 7.1.1, page 188]{Henry}, to conclude that $u(t)=v(t)$, for all $t\in[0,\tau]$. 
	
	To complete the proof consider $0<\theta<\beta$. Then for all $0< t\leq \tau$ we have 
	\begin{eqnarray*}
		\|u(t)\|_{X^{1+\theta}_{q}} & \le & \|E_{\alpha}(t)u_0\|_{X^{1+\theta}_{q}}+\|S_{\alpha}(t)u_1\|_{X^{1+\theta}_{q}}\\
		&+& \int_{0}^{t}\|R_{\alpha}(t-s)f(u(s))\|_{X^{1+\theta}_{q}}ds\\
		&\le& Mt^{-\alpha\theta}\|u_0\|_{X^{1}_{q}}+Mt^{1-\alpha\theta}\|u_1\|_{X^{1}_{q}}\\
		&+& CM\int_{0}^{t}(t-s)^{-1-\alpha(\theta-\beta)}\|u(s)\|_{X^{1}_{q}}^{\rho}ds\\
		&\le& Mt^{-\alpha\theta}\|u_0\|_{X^{1}_{q}}+Mt^{1-\alpha\theta}\|u_1\|_{X^{1}_{q}}\\
		&+& CM\sup_{0\le t\le \tau}\|u(t)\|_{X^{1}_{q}}^{\rho}t^{\alpha(\beta-\theta)}{\bf B}(\alpha(\beta-\theta),1).
	\end{eqnarray*}
	Therefore, $u:(0,\tau]\to X^{1+\theta}_{q}$ is well defined. With a similar argument to that used before we prove that $u\in C((0,\tau];X^{1+\theta}_{q})$. From the above estimate if $t>0$ we have that
	\begin{eqnarray*}
		t^{\alpha\theta}\|u(t)\|_{X^{1+\theta}_{q}} & \le & t^{\alpha\theta}\|E_{\alpha}(t)u_0\|_{X^{1+\theta}_{q}}+Mt\|u_1\|_{X^{1}_{q}}\\
		&+& CM\sup_{0\le t\le \tau}\|u(t)\|_{X^{1}_{q}}^{\rho}t^{\alpha\beta}{\bf B}(\alpha(\beta-\theta),1),
	\end{eqnarray*}
which implies that the right side of the above inequality goes to $0$ as $t\to 0^+$. Finally, if $u_0, u_1,w_0,w_1\in B_{r}(v_0)$ then
	\begin{eqnarray*}
		t^{\alpha\theta} \|u(t,u_0,u_1)-u(t,w_0,w_1)\|_{X^{1+\theta}_{q}} &\le& M\left(\|u_0-w_0\|_{X^{1}_{q}}+\|u_1-w_1\|_{X^{1}_{q}}\right)\\
		&+& \Gamma_{\theta}(t)\sup_{0\le t\le \tau}\|u(t,u_0,u_1)-u(t,w_0,w_1)\|_{X^{1}_{q}},
	\end{eqnarray*}
	where
	\begin{eqnarray*}
		\Gamma_{\theta}(t)&=&CMt^{\alpha\beta}{\bf B}(\alpha(\beta-\theta),1)\sup_{0\le t\le \tau}\|u(t,u_0,u_1)\|_{X^{1}_{q}}^{\rho-1}\\
		&+&CMt^{\alpha\beta}{\bf B}(\alpha(\beta-\theta),1)\sup_{0\le t\le \tau}\|u(t,w_0,w_1)\|_{X^{1}_{q}}^{\rho-1}.
	\end{eqnarray*}
	For $\theta=0$ we get
	\begin{eqnarray*}
		\|u(t,u_0,u_1)-u(t,w_0,w_1)\|_{X^{1}_{q}} &\le& M\left(\|u_0-w_0\|_{X^{1}_{q}}+\|u_1-w_1\|_{X^{1}_{q}}\right)\\
		&+& \frac{1}{4}\sup_{0\le t\le \tau}\|u(t,u_0,u_1)-u(t,w_0,w_1)\|_{X^{1}_{q}},
	\end{eqnarray*}
and then
	$$\sup_{0\le t\le \tau}\|u(t,u_0,u_1)-u(t,w_0,w_1)\|_{X^{1}_{q}} \le 2M\left(\|u_0-w_0\|_{X^{1}_{q}}+\|u_1-w_1\|_{X^{1}_{q}}\right)$$
	By taking $0\le\theta\le\theta_0<\beta$, these inequalities imply that
	\begin{eqnarray*}
		t^{\alpha\theta} \|u(t,u_0,u_1)-u(t,w_0,w_1)\|_{X^{1+\theta}_{q}} &\le& M\left(\|u_0-w_0\|_{X^{1}_{q}}+\|u_1-w_1\|_{X^{1}_{q}}\right)
	\end{eqnarray*}
	$$+\ 2M\Gamma_{\theta}(t)\left(\|u_0-w_0\|_{X^{1}_{q}}+\|u_1-w_1\|_{X^{1}_{q}}\right),$$
	that is
	$$t^{\alpha\theta} \|u(t,u_0,u_1)-u(t,w_0,w_1)\|_{X^{1+\theta}_{q}} \le c\left(\|u_0-w_0\|_{X^{1}_{q}}+\|u_1-w_1\|_{X^{1}_{q}}\right),$$
	where
	$$c=M(1+2\sup\{\Gamma_{\theta}(t):\ 0\le t\le \tau,\ 0\le\theta\le\theta_0 \}),$$
	which concludes the proof.
\end{proof}




Now, we prove a continuation result and a blow-up alternative for the mild solution guaranteed by Theorem \ref{existence1}. In particular, we ensure a maximal time of existence to this solution.

\begin{de} Let $u:[0,\tau]\to X^{1}_{q}$ be a mild solution of the problem (\ref{awave}). If $t_1>\tau$ and $v:[0, t_1]\to X^{1}_{q}$ is a mild solution of the problem (\ref{awave}), then we say that $v$ is a continuation of $u$ in $[0, t_1]$.
\end{de}

\begin{theorem}\label{continuation}
	Under the conditions of Theorem \ref{existence1}, consider $\tau>0$ and $u: [0,\tau ]\rightarrow X^{1}_{q}$ the mild solution of problem (\ref{awave}). Then, there exist $t_1>\tau$ and a unique continuation $u^*$ of $u$ in $[0,t_1]$.
\end{theorem}
\begin{proof}
	Fixe $0<\nu\leq 1$  and take $t_1> \tau $ so that for all $t\in [\tau, t_1]$,
	$$\|E_{\alpha}(t)u_0-E_{\alpha}(\tau)u_0\|_{X^{1}_{q}}<\dfrac{\nu}{4},\quad \|S_{\alpha}(t)u_1-S_{\alpha}(\tau )u_1\|_{X^{1}_{q}}< \dfrac{\nu}{4},$$
	$$\int_{0}^{\tau }\|[R_{\alpha}(t-s)-R_{\alpha}(\tau -s)]f(u(s))\|_{X^{1}_{q}}ds< \dfrac{\nu}{4}$$
	and
	$$MCRt^{\alpha\beta}\int_{\tau /t}^{1}(1-s)^{-1+\alpha\beta}ds< \dfrac{\nu}{4}, $$
	where $R:=\max\{(\|u(\tau )\|_{X^{1}_{q}}+\nu)^{\rho}, 2(\|u(\tau )\|_{X^{1}_{q}}+\nu)^{\rho-1}\}$.  Set
	$$K=\left\{v\in C([0,t_1]; X^{1}_{q}): \sup_{t\in[\tau ,t_1]}\|v(t)-u(\tau )\|_{X^{1}_{q}}\leq\nu, v(t)=u(t), t\in[0,\tau ]\right\}$$
	and define on $K$ the operator 
	$$(Tv)(t)=E_{\alpha}(t)u_0+S_{\alpha}(t)u_1+\int_{0}^{t}R_{\alpha}(t-s)f(v(s))ds.$$
	Note that, if $v\in K$ then $(Tv)(t)=u(t)$, for all $t\in[0,\tau ]$. The same computations performed there show that $Tv\in C([0,t_1]; X^{1}_{q})$  whenever $v\in K$. For $t\in[\tau ,t_1]$, we have
	$$\|(Tv)(t)-u(\tau )\|_{X^{1}_{q}}\leq\|E_{\alpha}(t)u_0-E_{\alpha}(\tau)u_0\|_{X^{1}_{q}} + \|S_{\alpha}(t)u_1-S_{\alpha}(\tau)u_1\|_{X^{1}_{q}}$$
	\begin{eqnarray*}	
		&+& \left\|\int_{0}^{t}R_{\alpha}(t-s)f(v(s))ds-\int_{0}^{\tau }R_{\alpha}(\tau -s)f(u(s))ds \right\|_{X^{1}_{q}}\\
		&\leq& \frac{3\nu}{4} +  \int_{\tau }^{t}\|R_{\alpha}(t-s)f(v(s))\|_{X^{1}_{q}}ds\\
		&\leq& \frac{3\nu}{4} + \int_{\tau }^{t}\|R_{\alpha}(t-s)f(v(s))\|_{X^{1}_{q}}ds\\
		&\leq& \frac{3\nu}{4}+MCRt^{\alpha\beta}\int_{\tau /t}^{1}(1-s)^{-1+\alpha\beta}ds\le \nu.
	\end{eqnarray*} 
	Then, it follows that
	$$\|(Tv)(t)-u(\tau )\|_{X^{1}_{q}}\leq \nu, \ \ \forall t\in [\tau ,t_1],$$
	whence $K$ is a $T$-invariant set. We claim that $T:K\rightarrow K$ is a contraction. In fact, if $v,w\in K$ and $t\in[\tau , t_1]$ then
	$$\|(Tv)(t)-(Tw)(t)\|_{X^{1}_{q}}=\left\|\int_{\tau }^{t}R_{\alpha}(t-s)[f(v(s))-f(w(s))]ds\right\|_{X^{1}_{q}}$$
	$$\leq MC\int_{\tau }^{t}(t-s)^{-1+\alpha\beta}(\|v(s)\|_{X^{1}_{q}}^{\rho-1}+\|w(s)\|_{X^{1}_{q}}^{\rho-1})\|v(s)-w(s)\|_{X^{1}_{q}}ds$$
	$$\leq MCRt^{\alpha\beta}\int_{\tau /t}^{1}(1-s)^{-1+\alpha\beta}ds\sup_{s\in[0,t_1]}\|v(s)-w(s)\|_{X^{1}_{q}}\leq \frac{1}{4}\sup_{s\in[0,t_1]}\|v(s)-w(s)\|_{X^{1}_{q}}.$$
	That is,
	$$\sup_{t\in[0,t_1]}\|(Tv)(t)-(Tw)(t)\|_{X^{1}_{q}}\leq\frac{1}{4}\sup_{t\in[0,t_1]}\|v(s)-w(s)\|_{X^{1}_{q}}.$$
	Hence $T$ is a strict contraction and has a unique fixed point $u^\ast\in K$ which is a continuation of $u$ in $[0,t_1]$. As before, the uniqueness can be shown by Gronwall Lemma.
\end{proof}

Next is our result of global existence or non-continuation by a blow-up. 

\begin{theorem}\label{Blow} Under the conditions of Theorem \ref{existence1}, let $u$ be the mild solution of problem (\ref{awave}) with maximal time of existence $\tau_{\max}>0$. Then either $\tau_{\max}=\infty$ or $$\displaystyle\limsup_{t\rightarrow \tau_{\max}^-}\|u(t)\|_{X^{1}_{q}}=\infty.$$
\end{theorem}
\begin{proof}
	Suppose that $\tau_{\max}<\infty$ and there exists $k>0$ such that $\|u(t)\|_{X^{1}_{q}}\leq k$, for all $t\in [0, \tau_{\max})$. Let $(t_n)_{n\in\N}\subset[0,\tau_{\max})$ such that $t_n\rightarrow \tau_{\max}$ as $n\rightarrow \infty$. Analogously to the proof of Theorem $\ref{existence1}$, it is shown that 	
	$$\int_{0}^{t_n}(t_n - s)^{\alpha\beta-1}\|\tilde{R}_{\alpha}(t_n-s)f(u(s))-\tilde{R}_{\alpha}(\tau_{\max}-s)f(u(s))\|_{X^{1}_{q}}ds\to0,$$
	as $n\rightarrow \infty$, 
	where $\tilde{R}_{\alpha}(t):= t^{1-\alpha\beta}R_{\alpha}(t)$, $t>0$. 		Suppose $t_n<t_m<\tau_{\max}$. Then,
	\begin{eqnarray*}
		\|u(t_n)-u(t_m)\|_{X^{1}_{q}}&\leq& \|E_{\alpha}(t_n)u_0-E_{\alpha}(t_m)u_0\|_{X^{1}_{q}}\\
		&+& \|S_{\alpha}(t_n)u_1-S_{\alpha}(t_m)u_1\|_{X^{1}_{q}}\\                              
		&+&\int_{0}^{t_n}\|[R_{\alpha}(t_n-s)-R_{\alpha}(t_m-s)]f(u(s))\|_{X^{1}_{q}}ds\\
		&+& \int_{t_n}^{t_m}\|R_{\alpha}(t_m-s)f(u(s))\|_{X^{1}_{q}}ds.
	\end{eqnarray*}
	By strong continuity of $(E_{\alpha}(t))_{t\ge0}$ and $(S_{\alpha}(t))_{t\ge0}$, we have that the two first terms of the right side in the above inequality go to zero as $m, n\rightarrow\infty.$ On the other hand
	\begin{eqnarray*}
		\int_{t_n}^{t_m}\|R_{\alpha}(t_m-s)f(u(s))\|_{X^{1}_{q}}ds &\leq& M \int_{t_n}^{t_m}(t_m-s)^{\alpha\beta-1}\|f(u(s))\|_{X_{\beta}}ds\\
		&\leq& \frac{MC}{\alpha\beta}k^{\rho}(t_m-t_n)^{\alpha\beta}\rightarrow 0,
	\end{eqnarray*}
	as $m,n\rightarrow\infty$. Furthermore, set
	
	$$I=\int_{0}^{t_n}\|[R_{\alpha}(t_n-s)-R_{\alpha}(t_m-s)]f(u(s))\|_{X^{1}_{q}}ds.$$
	Then
	\begin{eqnarray*}
		&&	I \leq \int_{0}^{t_n}[(t_m-s)^{\alpha\beta-1}-(t_n-s)^{\alpha\beta-1}]\|\tilde{R}_{\alpha}(t_n-s)f(u(s))\|_{X^{1}_{q}}ds\\
		&+& \!\!\!\!\!\!\int_{0}^{t_n}(t_n-s)^{\alpha\beta-1}\|\tilde{R}_{\alpha}(t_m-s)f(u(s))-\tilde{R}_{\alpha}(t_n-s)f(u(s))\|_{X^{1}_{q}}ds\\
		&\leq&\frac{MC}{\alpha\beta}k^{\rho}[t_m^{\alpha\beta}-(t_m-t_n)^{\alpha\beta}-t_n^{\alpha\beta}]
	\end{eqnarray*}
	$$+ \int_{0}^{t_n}(t_n - s)^{\alpha\beta-1}\|\tilde{R}_{\alpha}(t_n-s)f(u(s))-\tilde{R}_{\alpha}(\tau_{\max}-s)f(u(s))\|_{X^{1}_{q}}ds$$
	$$+ \int_{0}^{t_m}(t_m - s)^{\alpha\beta-1}\|\tilde{R}_{\alpha}(t_m-s)f(u(s))-\tilde{R}_{\alpha}(\tau_{\max}-s)f(u(s))\|_{X^{1}_{q}}ds$$
	Hence, if $m,n\rightarrow \infty$ it follows that $I\to0$. This computation shows that  $\{u(t_n)\}_{n\in\mathbb{N}}$ is a Cauchy  sequence in $X^{1}_{q}$ and thus there exists $\tilde{u}\in X^{1}_{q}$ such that 
	$\lim_{n\rightarrow\infty}\|u(t_n)-\tilde{u}\|_{X^{1}_{q}}=0.$
	With this, we can extend $u$ to $[0,\tau_{\max}]$ obtaining the equality  
	$$u(t)=E_{\alpha}(t)u_0+S_{\alpha}(t)u_0+ \int_{0}^{t}R_{\alpha}(t-s)f(u(s))ds$$
	for all $t\in[0,\tau_{\max}]$. By Theorem $\ref{continuation}$, we can extend the solution to a bigger interval, which is a contradiction with the maximality of  $\tau_{\max}$.
\end{proof}

We close this work with some illustrations of our main results. Consider the particular problem 
\begin{eqnarray}\label{exwave}	 					
	\left\{\begin{array}{lrr}
		\partial_t^{\alpha}u  = \Delta u  + u\sqrt{ |u|}, \ \mbox{in}\ (0,\infty)\times\Omega,\\	
		u  =  0, \ \mbox{in} \  [0,\infty)\times \partial\Omega,\\
		u(0,x)=u_0(x),\ u_t(0,x)=u_1(x),\ \mbox{in}\ \ \Omega,\\
	\end{array} \right.	
\end{eqnarray}
where $\Omega \subset\mathbb{R}^3$ is a bounded smooth domain.  We apply our main results in the $L^2(\Omega)$-setting.  In this situation, we can choose the derivation order $\alpha$ as any real number in $(1,2)$. Indeed, it follows from Theorems \ref{existence1}, \ref{continuation} and \ref{Blow} that, for any initial data $u_0,u_1\in X^{1}_{2}=L^2(\Omega)$, there exists a unique mild solution $u\in C([0,\tau_{\max});X^{1}_{2})$ to problem \eqref{exwave}, given by
$$u(t,x)= E_{\alpha}(t)u_0(x)+S_{\alpha}(t )u_1(x)+\int_{0}^{t}R_{\alpha}(t-s)u(s,x)\sqrt{|u(s,x)|}ds, $$
and 
$$\limsup_{t\to\tau_{\max}^{-}}\int_{\Omega}|u(t,x)|^2dx=+\infty,$$ 
if $\tau_{\max}<+\infty$. Furthermore, for all $\tau<\tau_{\max}$ and $0\le \theta <\frac{5}{8}$ it follows that $u\in C([0,\tau];X^{1+\theta}_{2})$, and if $\theta>0$, then 
$$\lim_{t\to0^{+}}\|u(t,\cdot)\|_{X^{1+\theta}_{2}}=0.$$
Finally, this solution depends continuously from the initial data.

\noindent{\bf Statements and Declarations:} 
\begin{itemize}
	\item Bruno de Andrade is partially supported by CNPQ/Brazil under grant 310384/2022-2.
\end{itemize}
			

		\end{document}